\documentclass[a4paper,12pt]{amsart}
\usepackage{amssymb}
\usepackage{ifthen}
\usepackage{graphicx}
\usepackage{amsmath}
\usepackage[T1]{fontenc} 
\usepackage{color,soul}
\nonstopmode \numberwithin{equation}{section}
\theoremstyle{plain}
\newtheorem{thm}{Theorem}[section]
\newtheorem{cor}{Corollary}[section]
\newtheorem{lem}{Lemma}[section]
\newtheorem{prop}{Proposition}

\newtheorem{conj}{Conjecture}
\newtheorem{innercustomthm}{Theorem}
\newenvironment{customthm}[1]
  {\renewcommand\theinnercustomthm{#1}\innercustomthm}
  {\endinnercustomthm}

\theoremstyle{definition}
\newtheorem{defn}{Definition}[section]

\newtheorem{prob}{Problem}
\newtheorem{rem}{Remark}[section]

\newcounter{minutes}
\divide\time by 60
\newcounter{hours}
\multiply\time by 60
\addtocounter{minutes}{-\time}

\newcounter {own}
\def\theown {\thesection       .\arabic{own}}

\newenvironment{pf}[1][]{%
 \vskip 3mm
 \noindent
 \ifthenelse{\equal{#1}{}}%
  {{\slshape Proof. }}%
  {{\slshape #1.} }%
 }%
{\qed\bigskip}

\newcounter{alphabet}

\def\be{\begin{equation}}
\def\ee{\end{equation}}

\newcommand{\bee}{\begin{enumerate}}
\newcommand{\eee}{\end{enumerate}}

\newcommand{\blem}{\begin{lem}}
\newcommand{\elem}{\end{lem}}
\newcommand{\bthm}{\begin{thm}}
\newcommand{\ethm}{\end{thm}}
\newcommand{\bcor}{\begin{cor}}
\newcommand{\ecor}{\end{cor}}
\newcommand{\beg}{\begin{examp}}
\newcommand{\eeg}{\end{examp}}
\newcommand{\begs}{\begin{examples}}
\newcommand{\eegs}{\end{examples}}
\newcommand{\bdefn}{\begin{defn}}
\newcommand{\edefn}{\end{defn}}
\newcommand{\bprob}{\begin{prob}}
\newcommand{\eprob}{\end{prob}}
\newcommand{\bei}{\begin{itemize}}
\newcommand{\eei}{\end{itemize}}

\newcommand{\bcon}{\begin{conj}}
\newcommand{\econ}{\end{conj}}
\newcommand{\bcons}{\begin{conjs}}
\newcommand{\econs}{\end{conjs}}
\newcommand{\bprop}{\begin{prop}}
\newcommand{\eprop}{\end{prop}}
\newcommand{\br}{\begin{rem}}
\newcommand{\er}{\end{rem}}
\newcommand{\brs}{\begin{rems}}
\newcommand{\ers}{\end{rems}}
\newcommand{\bo}{\begin{obser}}
\newcommand{\eo}{\end{obser}}
\newcommand{\bos}{\begin{obsers}}
\newcommand{\eos}{\end{obsers}}
\newcommand{\bpf}{\begin{pf}}
\newcommand{\epf}{\end{pf}}
\newcommand{\ba}{\begin{array}}
\newcommand{\ea}{\end{array}}
\newcommand{\beq}{\begin{eqnarray}}
\newcommand{\beqq}{\begin{eqnarray*}}
\newcommand{\eeq}{\end{eqnarray}}
\newcommand{\eeqq}{\end{eqnarray*}}

\begin{document}

\title{Bohr type inequality for  {C}es\'{a}ro and {B}ernardi integral operator on simply connected domain}
\author{Vasudevarao Allu}
\address{Vasudevarao Allu,
School of Basic Sciences ,
Indian Institute of Technology Bhubaneswar,
Bhubaneswar-752050, Odisha, India.}
\email{avrao@iitbbs.ac.in}

\author{Nirupam Ghosh}
\address{Nirupam Ghosh,
Stat-Math Unit,
Indian Statistical Institute, Bnagalore,
Bangalore- 560059, Karnataka, India.}
\email{nirupamghoshmath@gmail.com}

\subjclass[2010]{Primary 30C45, 30C50, 40G05}
\keywords{Analytic functions, Bohr radius, {C}es\'{a}ro operator, {B}ernardi integral}

\def\thefootnote{}
\footnotetext{ {\tiny File:~\jobname.tex,
printed: \number\year-\number\month-\number\day,
          \thehours.\ifnum\theminutes<10{0}\fi\theminutes }
} \makeatletter\def\thefootnote{\@arabic\c@footnote}\makeatother

\begin{abstract}
In this article, we study the Bohr type  inequality for {C}es\'{a}ro operator and {B}ernardi integral  operator acting on the space of analytic functions defined on a simply connected domain containing the unit disk $\mathbb{D}$.
\end{abstract}

\maketitle
\pagestyle{myheadings}
\markboth{  Vasudevarao Allu and  Nirupam Ghosh }{Bohr type inequality of {C}es\'{a}ro and {B}ernardi integral operator on simply connected domain}

\section{Introduction}

Let $D(a; r)= \{z \in \mathbb{C}: |z - a| < r\}$ and $\mathbb{D} = D (0; 1)$ be the  unit disk in the complex plane $\mathbb{C}$. For a simply connected domain $\Omega$ containing $\mathbb{D}$, let $\mathcal{H}(\Omega)$ denote the class of analytic functions on $\Omega$, and let $$\mathcal{B}(\Omega)= \{f\in \mathcal{H}(\Omega) : f(\Omega) \subset \overline{\mathbb{D}} \}.$$ The Bohr radius \cite{Fournier-Ruscheweyh-2010} for the family $\mathcal{B}(\Omega)$ is defined to be the positive real number $R_{\Omega} \in (0, 1)$ given by
\begin{equation*}
R_{\Omega} = \sup \{r \in (0, 1): M_f(r) \leq 1 ~~~~ \mbox{for}~~~ f(z) = \sum_{n = 0}^{\infty} a_n z^n \in \mathcal{B}(\Omega), z\in \mathbb{D}\},
\end{equation*}
where $ M_f(r) = \sum_{n = 0}^{\infty} |a_n |r^n $  with $|z| = r $ is the majorant series associated with $f \in \mathcal{B}(\Omega)$ in $\mathbb{D}$. If $\Omega = \mathbb{D}$, then it is well-known that $R_{\mathbb{D}} = 1/3$, and it is descried precisely as follows:

\begin{customthm}{A}\label{niru-vasu-P8-theorem001}
If $f \in \mathcal{B}(\mathbb{D})$, then $M_f(r) \leq 1$ for $0\leq r \leq 1/3$. The number $1/3$ is best possible.
\end{customthm}

The inequality $M_f(r) \leq 1$ for $f\in \mathcal{B}(\mathbb{D})$, fails to hold for any $r> 1/3$. This can be seen by considering the function $\phi_{a}(z) = (a - z)/ (1 - a z)$ and take $a\in (0, 1)$ such that $a$ is sufficiently close to $1$.
Theorem \ref{niru-vasu-P8-theorem001} was originally obtained by H. Bohr \cite{Bohr-1914} in $1914$  for $0\leq r\leq1/6$. The optimal value $1/3$, which is called the Bohr radius for the unit disk, has been established by M. Riesz, I Schur and F. W. Weiner (see \cite{Sidon-1927}, \cite{Tomic-1962}). Over the past two decades there has been significant interest on the Bohr type inequalities, one may  see articles \cite{Abu-2010, abu-2011,  abu-2014, Ali-2017, alkhaleefah-2019, Ismagilov-Ponnusamy-2020, Kayumov-Ponnusamy-2017,  Kayumov-Ponnusamy-2019} and the references therein.\\

\bigskip
Besides the Bohr radius, there is  a notion of Rogosinski radius \cite{landau-Gaier-1986, Rogosinski-1923} which is described as follows: Let $f(z) = \sum_{n = 0 }^{\infty}a_n z^n \in \mathcal{B}(\mathbb{D})$ and the corresponding partial sum of $f$ is defined by $S_N(z) = \sum_{n = 0 }^{N -1}a_n z^n$.  Then, for every $N \geq 1$, we have
$\left|S_N(z)\right| < 1$ in the disk $|z| < 1/2$ and the radius $1/2$ is sharp. Motivated by Rogosinski radius, Kayumov and Ponnusamy \cite{Kayumov-Ponnusamy-arxiv} have considered the Bohr-Rogosinski sum as
$$
R^{f}_{N} := |f(z)| + \sum_{n = N}^{\infty}|a_n||z|^n
$$
for $f \in \mathcal{B}(\mathbb{D})$ and defined the Bohr-Rogosinski radius as the largest number $r> 0$ such that $R^{f}_{N} \leq 1$ for $|z|< r$. For a significant and an extensive research in the direction of Bohr-Rogosinki radius, we referred to  \cite{Ismagilov-Ponnusamy-2020, Kayumov-Ponnusamy-2018, Kayumov-Ponnusamy-2021} and the references therein.

\bigskip
A natural question arises ``Can we extend the Bohr type  inequality for certain complex integral operators defined on various function spaces?" The idea has been initiated for the classical {C}es\'{a}ro operator in \cite{Kayumov-Ponnusamy-2020, Kayumov-Ponnusamy-2021} and for {B}ernardi integral operator in \cite{Kumar-Sahoo-2021}. In \cite{Kayumov-Ponnusamy-2020,Kayumov-Ponnusamy-2021,  Kumar-Sahoo-2021} the  authors have studied the Bohr type and Bohr-Rogosinski type inequalities for {C}es\'{a}ro  operator and {B}ernardi integral operator defined on $\mathcal{B}(\mathbb{D})$.

{C}es\'{a}ro operator and its various generalizations have been extensively studied. For example, the boundness and compactness of {C}es\'{a}ro operator on different function spaces has been well studied. In the classical setting, for an analytic  function $f(z) = \sum_{n = 0}^{\infty}a_n z^n$ on the unit disk $\mathbb{D}$, the {C}es\'{a}ro operator is defined by \cite{Hardy-Littlewoow-1931} (see also \cite{Halmos-1965, Stempak-1994})
\begin{equation}\label{niru-vasu-P8-eq000a}
\mathcal{C}f(z) := \sum_{n = 0}^{\infty}\frac{1}{n + 1}\left( \sum_{k = 0}^{n}a_k\right) z^n = \int_{0}^{1} \frac{f(t z)}{1 - t z}\, dt.
\end{equation}
It is not difficult to show that for $f \in \mathcal{B}(\mathbb{D})$,
$$
\left|\mathcal{C}f(z) \right| = \left|\sum_{n = 0}^{\infty}\frac{1}{n + 1}\left( \sum_{k = 0}^{n}a_k\right) z^n \right| \leq \frac{1}{r}\ln{\frac{1}{1 - r}} \quad \mbox{for}\quad |z| = r.
$$
In 2020,  Kayumov {\it et al.} \cite{Kayumov-Ponnusamy-2020} have been established the following Bohr type inequality for {C}es\'{a}ro operator.
\begin{customthm}{B}\label{niru-vasu-P8-theorem002}
If $f \in \mathcal{B}(\mathbb{D})$ and $f(z) = \sum_{n = 0}^{\infty} a_n z^n$, then
$$
\mathcal{C}_{f}(r) = \sum_{n = 0}^{\infty}\frac{1}{n + 1}\left( \sum_{k = 0}^{n}|a_k|\right) r^n  \leq \frac{1}{r}\ln{\frac{1}{1 - r}}
$$
for $|z| = r \leq R,$  where $R = 0.5335\ldots$ is the positive root of the equation
$$
2 x = 3(1 -x)\ln{\frac{1}{1-x}}.
$$
The number $R$ is the best possible.
\end{customthm}

For an analytic function $f(z) = \sum_{n = m}^{\infty}a_n z^n$ on the unit disk $\mathbb{D}$, the {B}ernardi integral operator (see \cite{Miller-Mocanu-2020}) is defined by
$$
\mathcal{L}_{\beta}f(z) : = \frac{1 + \beta}{z^\beta}\int_{0}^{z} f(\xi) \xi^{\beta - 1} \, d \xi =(1 + \beta)\sum_{n = m}^{\infty}\frac{a_n}{\beta + n} z^n,
$$
where $\beta > -m$ and $m \geq 0$ is an integer.
It is worth to mention that  for each $|z| = r \in[0, 1)$, the integral representation for {B}ernardi integral operator yields the following for $f\in \mathcal{B}(\mathbb{D})$:
$$
\left|\mathcal{L}_{\beta}f(z) \right| = \left|(1 + \beta)\sum_{n = 0}^{\infty}\frac{a_n}{\beta + n} z^n \right|\leq (1 + \beta)\frac{r^m}{m +\beta},
$$
which is equivalent to the following expression
$$
\left|\sum_{n = m}^{\infty}\frac{a_n}{\beta + n} z^n \right|\leq \frac{r^m}{m +\beta}.
$$
Recently, Kumar and Sahoo \cite{Kumar-Sahoo-2021} have studied the following Bohr type inequality for  {B}ernardi integral operator.

\begin{customthm}{C}\label{niru-vasu-P8-theorem003}
Let $\beta > -m$. If $f(z) = \sum_{n = m}^{\infty} a_n z^n \in \mathcal{B}(\mathbb{D})$, then
$$
\sum_{n = m}^{\infty}\frac{|a_n|}{\beta + n} |z|^n \leq \frac{r^m}{m +\beta}
$$
for $|z| = r \leq R(\beta)$. Here $R(\beta)$ is the positive root of the equation
$$
\frac{x^m}{m + \beta} - 2 \sum_{n = m + 1}^{\infty} \frac{x^n}{n + \beta}=0
$$
that cannot be improved.
\end{customthm}


The maim aim of this paper is to find the sharp Bohr type inequality for {C}es\'{a}ro operator and {B}ernardi integral operator for functions in  the class
$\mathcal{B}(\Omega_{\gamma})$, where
$$
\Omega_{\gamma}:= \bigg \{z \in \mathbb{C} : \left|z + \frac{\gamma}{1 - \gamma} \right| < \frac{1}{1 - \gamma}\bigg\} \quad \mbox{for} \quad 0 \leq \gamma < 1.
$$
Clearly the unit disk $\mathbb{D}$ is always a subset of $\Omega_{\gamma}$. In 2010, Fournier and Ruscheweyh \cite{Fournier-Ruscheweyh-2010} extended the Bohr's inequality
for functions in  $\mathcal{B}(\Omega_{\gamma})$.

\bigskip


The following lemma by  Evdoridis {\it et al.} \cite{Evdoridis-Ponnusay-Rasila-2021} plays a crucial rule to prove our main results.

\begin{lem} \label{niru-vasu-P8-lem001}\cite{Evdoridis-Ponnusay-Rasila-2021}
For $\gamma\in [0, 1)$, let
$$
\Omega_{\gamma}:= \bigg \{z \in \mathbb{C} : \left|z + \frac{\gamma}{1 - \gamma} \right| < \frac{1}{1 - \gamma}\bigg\},
$$
and let $f$ be an analytic function in $\Omega_{\gamma}$, bounded by $1$, with the series representation $f(z) = \sum_{n = 0}^{\infty} a_n z^n$ in  $\mathbb{D}$. Then,
$$
|a_n| \leq \frac{1 - |a_0|^2}{ 1 + \gamma} \quad \mbox{for} \quad n\geq 1.
$$
\end{lem}

\section{Main results}

We state and prove our first main result.

\begin{thm}\label{niru-vasu-P8-theorem005}
For $0 \leq \gamma < 1$, let $f \in \mathcal{B}(\Omega_{\gamma})$ with $f(z) = \sum_{n = 0}^{\infty} a_n z^n$ in $\mathbb{D}$. Then we have
\begin{equation}\label{niru-vasu-P8-eq000}
\mathcal{C}_{f}(r) = \sum_{n = 0}^{\infty}\frac{1}{n + 1}\left( \sum_{k = 0}^{n}|a_k|\right) r^n  \leq \frac{1}{r}\ln{\frac{1}{1 - r}}\quad \mbox{for}~~~ |z|= r \leq R_\gamma
\end{equation}
where $R_\gamma$ is the positive root of
$$
(3 + \gamma)(1 - x)\ln{\frac{1}{1 - x}} = 2 x.
$$
The number $R_\gamma$ is the best possible.
\end{thm}

\begin{proof}
Let $|a_0| = a < 1$. A simple computation of the {C}es\'{a}ro operator in (\ref{niru-vasu-P8-eq000a})  shows that
\begin{align}\label{niru-vasu-P8-eq001}
\mathcal{C}_{f}(r) & = \sum_{n = 0}^{\infty}\frac{1}{n + 1}\left( \sum_{k = 0}^{n}|a_k|\right) r^n \\\nonumber
& = a \left(1 + \frac{r}{2} + \frac{r^2}{3}+ \cdots\right) + \sum_{n = 1}^{\infty}\frac{1}{n + 1}\left( \sum_{k = 1}^{n}|a_k|\right) r^n\\ \nonumber
& = \frac{a}{r}\ln{\frac{1}{1 - r}} +  \sum_{n = 1}^{\infty}\frac{1}{n + 1}\left( \sum_{k = 1}^{n}|a_k|\right) r^n.
\end{align}
Using Lemma \ref{niru-vasu-P8-lem001} in  (\ref{niru-vasu-P8-eq001}) we obtain the following estimation for the  {C}es\'{a}ro operator:

\begin{align}\label{niru-vasu-P8-eq005}
\mathcal{C}_{f}(r) & \leq \frac{a}{r}\ln{\frac{1}{1 - r}} +  \frac{1 - a^2}{1 + \gamma}\sum_{n = 1}^{\infty}\frac{n}{n + 1} r^n\\ \nonumber
&= \frac{a}{r}\ln{\frac{1}{1 - r}} +  \frac{1 - a^2}{1 + \gamma} \bigg( \frac{1}{1 - r} - \frac{1}{r}\ln {\frac{1}{1 - r}} \bigg).
\end{align}
Let $$P_{\gamma, r}(a) = \frac{a}{r}\ln{\frac{1}{1 - r}} +  \frac{1 - a^2}{1 + \gamma} \bigg( \frac{1}{1 - r} - \frac{1}{r}\ln {\frac{1}{1 - r}} \bigg).$$ Then twice differentiation of $P_{\gamma, r}$ with respect to $a$ shows that
\begin{equation*}
P''_{\gamma, r} (a)= \frac{-2}{1 + \gamma} \left( \frac{1}{1 - r} - \frac{1}{r}\ln {\frac{1}{1 - r}} \right) \leq 0
\end{equation*}
for all $a\in [0, 1)$ and for all $r \in [0, 1)$. Therefore, $P'_{\gamma, r}$ is a decreasing function  and hence we obtain
\begin{equation}\label{niru-vasu-P8-eq005aa}
P'_{\gamma, r} (a) \geq P'_{\gamma, r} (1)=  \frac{1}{r(1 - r)(1 + \gamma)}\left(-2r + (3 + \gamma)(1 - r)\ln \frac{1}{1 - r}\right) \geq 0
\end{equation}
for all $r \leq R_\gamma$. Thus, $P_{\gamma, r}(a)$ is  increasing  for $r \leq R_\gamma$ and for all $\gamma\in [0, 1)$ and hence

\begin{equation}\label{niru-vasu-P8-eq005a}
P_{\gamma, r}(a) \leq P_{\gamma, r} (1) = \frac{1}{r}\ln{\frac{1}{1 - r}} \quad \mbox{for all}\quad r \leq R_\gamma.
\end{equation}
Therefore, the  desired inequality (\ref{niru-vasu-P8-eq000}) follows (\ref{niru-vasu-P8-eq005a}).\\[3mm]

Now we show that the radius $R_\gamma$ cannot be improved. In order to prove the sharpness of the result, we consider  the function $G: \Omega_{\gamma} \rightarrow \mathbb{D} $ defined by $G(z) = (1 - \gamma) z + \gamma$ and $\psi: \mathbb{D}\rightarrow \mathbb{D}$ defined by
$$
\psi(z) = \frac{a - z}{1 - a z}
$$
for $a\in(0, 1)$. Then $f_\gamma = \psi \circ G$ maps $\Omega_\gamma$ univalently onto $\mathbb{D}$. A simple computation shows that
\begin{equation*}
f_\gamma (z) = \frac{a - \gamma -(1 - \gamma)z}{1 - a \gamma - a (1 - \gamma)z } = A_0 - \sum_{n = 1}^{\infty} A_n z^n,~~~ z\in \mathbb{D},
\end{equation*}
where $a \in (0, 1)$,
\begin{equation}\label{niru-vasu-P8-eq006}
A_0 = \frac{a - \gamma}{1 - a \gamma}~~~ \mbox{ and}~~~ A_n = \frac{1 - a^2}{a (1 - a \gamma)} \bigg( \frac{a (1 - \gamma)}{1 - a \gamma}\bigg)^n.
\end{equation}
For a given $\gamma \in [0, 1)$,  let $a > \gamma$. Then the Ces\'{a}ro operator on $f_\gamma$ shows that
\begin{align}\label{niru-vasu-P8-eq010}
\mathcal{C}f_\gamma(r) & = \sum_{n = 0}^{\infty}\frac{1}{n + 1}\left(\sum_{k = 0}^{n} |A_k| \right) r^n\\\nonumber
& = \frac{A_0}{r}\ln\frac{1}{1 -r} + \sum_{n = 1}^{\infty} \frac{1}{n + 1} \left(\sum_{k = 1}^{n} |A_k| \right) r^n.
\end{align}
By substituting $A_0$ and $A_n$ for $n \geq 1$ in (\ref{niru-vasu-P8-eq010}), we obtain
\begin{align}\label{niru-vasu-P8-eq015}
\mathcal{C}f_\gamma(r) & = \frac{a - \gamma}{r(1 - a \gamma)}\ln \frac{1}{1 - r} + \sum_{n = 1}^{\infty} \frac{1}{n + 1} \left(\sum_{k = 1}^{n} \frac{1 - a^2}{a (1 - a \gamma)} \left(\frac{a(1 - \gamma)}{1 - a \gamma}\right)^k \right) r^n\\\nonumber
& = \frac{a - \gamma}{r(1 - a \gamma)}\ln \frac{1}{1 - r}  + \frac{(1 - a^2)(1 - \gamma)}{(1 - a \gamma)^2}\sum_{n = 1}^{\infty} \frac{1}{n + 1} \left(\sum_{k = 1}^{n}  \left(\frac{a(1 - \gamma)}{1 - a \gamma}\right)^{k-1} \right) r^n\\\nonumber
& = \frac{a - \gamma}{r(1 - a \gamma)}\ln \frac{1}{1 - r}  + \frac{(1 + a )(1 - \gamma)}{(1 - a \gamma)}\sum_{n = 1}^{\infty} \frac{1}{n + 1} \left( 1 - \frac{a ^n (1 - \gamma)^n}{(1 - a \gamma)^n} \right)r^n.
\end{align}
Further simplification of  (\ref{niru-vasu-P8-eq015}) shows that
\begin{align*}
\mathcal{C}f_\gamma(r) & = \frac{a - \gamma}{r(1 - a \gamma)}\ln \frac{1}{1 - r}  + \frac{(1 + a )(1 - \gamma)}{r (1 - a \gamma)}\ln \frac{1}{1 - r} - \frac{1 + a}{a r}\ln \left(\frac{1}{1 - \frac{(1 - \gamma) a r}{1 - a \gamma} }\right)\\
& = \frac{1}{r}\ln\frac{1}{1 - r} + \frac{(1 - a)}{(1 - a \gamma)}\frac{2 r + (3 + \gamma)(1 - r)\ln (1 - r)}{r (1 - r)} + D_{a, \gamma}(r),
\end{align*}
where
\begin{align*}
D_{a, \gamma}(r) & = \frac{(3 - a)- \gamma(1 + a)}{1 - a \gamma} - 2 \frac{(1 - a)}{(1 - a \gamma)(1 - r)}  - \frac{1 + a}{a r}\ln \left(\frac{1}{1 - \left(\frac{(1 - \gamma) a r}{1 - a \gamma} \right)}\right)\\
& = \sum_{n = 1}^{\infty}\left( \frac{(3 - a)- \gamma(1 + a)}{1 - a \gamma} - 2 \frac{(1 - a)}{(1 - a \gamma)} - \frac{a^n (1 + a )(1 - \gamma)^{n + 1}}{(1 - a \gamma)^{n + 1}}\right) r^n\\
& = O((1 - a)^2) ~~~ \mbox{as}~~~ a \rightarrow 1.
\end{align*}
From (\ref{niru-vasu-P8-eq005aa}), we obtain $\left(-2r + (3 + \gamma)(1 - r)\ln \frac{1}{1 - r}\right) \geq 0$
for all $r \leq R_\gamma$ and hence
$$\frac{2 r + (3 + \gamma)(1 - r)\ln (1 - r)}{r (1 - r)} > 0 \quad \mbox{for} ~~~ r > R_\gamma.$$
These two facts show that the number cannot be improved.

\end{proof}

\begin{rem}
Since for $\gamma = 0$, the domain $\Omega_\gamma$ reduces to the unit disk $\mathbb{D}$, Theorem \ref{niru-vasu-P8-theorem002} is a direct consequence of Theorem  \ref{niru-vasu-P8-theorem005} when $\gamma = 0$.
\end{rem}

In the next result we study the Bohr type inequality for {B}ernardi integral operator for the class of analytic functions defined on $\Omega_\gamma$. 

\begin{thm}
For $0 \leq \gamma < 1$, let $f \in \mathcal{B}(\Omega_{\gamma})$ with $f(z) = \sum_{n = 0}^{\infty} a_n z^n$ in $\mathbb{D}$. Then for $\beta > 0$
\begin{equation*}
\sum_{n = 0}^{\infty}\frac{|a_n|}{n + \beta} r^n \leq \frac{1}{\beta} ~~~ \mbox{for}~~~ r \leq R_{\gamma, \beta}
\end{equation*}
where $R_{\gamma, \beta}$ is the positive root of
$$
\frac{1}{\beta} = \frac{2 }{1 + \gamma} \sum_{n = 1}^{\infty} \frac{r^n}{n + \beta}.
$$
The number $R_{\gamma, \beta}$ is the best  possible.
\end{thm}

\begin{proof}
Let $|a_0| = a < 1$. Then
\begin{align}\label{niru-vasu-P8-eq020}
\sum_{n = 0}^{\infty}\frac{|a_n|}{n + \beta} r^n = \frac{a}{\beta} + \sum_{n = 1}^{\infty}\frac{|a_n|}{n + \beta} r^n.
\end{align}
In view of Lemma \ref{niru-vasu-P8-lem001} and  (\ref{niru-vasu-P8-eq020}), we  obtain
\begin{align*}
\sum_{n = 0}^{\infty}\frac{|a_n|}{n + \beta} r^n \leq \frac{a}{\beta}  + \frac{1 - a^2}{1 + \gamma}\sum_{n = 1}^{\infty}\frac{r^n}{n + \beta}.
\end{align*}
Let $$\Phi_{\gamma, \beta} (a) = \frac{a}{\beta}  + \frac{1 - a^2}{1 + \gamma}\sum_{n = 1}^{\infty}\frac{r^n}{n + \beta}.$$
Then twice differentiation of $\Phi_{\gamma, \beta}$ with respect to $a$ shows that
$$\Phi''_{\gamma, \beta} (a) = -\frac{2}{1 + \beta}\sum_{n = 1}^{\infty}\frac{r^n}{n + \gamma} \leq 0$$ for all $a \in [0, 1]$ and for all $r \in [0, 1)$. This implies that $\Phi'_{\gamma, \beta}$ is decreasing and
\begin{equation}\label{niru-vasu-P8-eq021}
\Phi'_{\gamma, \beta} (a) \geq \Phi'_{\gamma, \beta} (1) = \left(\frac{1}{\beta}  - \frac{2}{1 + \gamma} \sum_{n = 1}^{\infty}\frac{r^n}{(n + \beta)}\right) \geq 0
\end{equation}
for $r \leq R_{\gamma, \beta}$. Hence $\Phi_{\gamma, \beta} (a)$ is increasing for $r\leq R_{\gamma, \beta} $. Therefore for all $a \in [0, 1]$,
$$
\Phi_{\gamma, \beta}(a) \leq \Phi_{\gamma, \beta}(1) = \frac{1}{\beta} \quad \mbox{for}~~~ r \leq R_{\gamma, \beta}
$$
and hence
\begin{equation*}
\sum_{n = 0}^{\infty}\frac{|a_n|}{n + \beta} r^n \leq \frac{1}{\beta} \quad  \mbox{for}~~~ r \leq R_{\gamma, \beta}.
\end{equation*}
We now show that $R_{\gamma, \beta}$ cannot be improved. In order to do this, consider the function
\begin{equation*}
f_\gamma (z) = \frac{a - \gamma -(1 - \gamma)z}{1 - a \gamma - a (1 - \gamma)z } = A_0 - \sum_{n = 1}^{\infty} A_n z^n, \quad z\in \mathbb{D},
\end{equation*}
where $a \in (0, 1)$, and $A_n (n \geq 0)$ are given by (\ref{niru-vasu-P8-eq006}). For a given $\gamma \in [0, 1)$,  let $a > \gamma$. Then for $\gamma\in [0, 1)$ and $\beta \geq 1$, we have
\begin{align}\label{niru-vasu-P8-eq025}
\sum_{n = 0}^{\infty}\frac{|A_n|}{n + \beta} r^n & = \frac{A_0}{\beta} + \sum_{n = 1}^{\infty}\frac{|A_n|}{n + \beta} r^n \\\nonumber
& = \frac{a - \gamma}{(1 - a \gamma) \beta} + \sum_{n = 1}^{\infty}\frac{1 - a^2}{a (1 - a \gamma)} \left(\frac{a(1 - \gamma)}{1 - a \gamma}\right)^n \frac{r^n}{n + \beta}.
\end{align}
By a simple computation, from (\ref{niru-vasu-P8-eq025}), we obtain
\begin{align*}
\sum_{n = 0}^{\infty}\frac{|A_n|}{n + \beta} r^n = \frac{1 }{\beta} - (1 - a) \left( \frac{1}{\beta} - \frac{2 }{1 + \gamma} \sum_{n = 1}^{\infty} \frac{r^n}{n + \beta}\right) + M_{a, \gamma, \beta}(r),
\end{align*}
where
\begin{align*}
M_{a, \gamma, \beta}(r) & = -\frac{1}{\beta} + \frac{a - \gamma}{(1 - a \gamma)\beta}\\
& + \frac{(1 - a^2)}{a(1 - a \gamma)} \sum_{n = 1}^{\infty}\left(\frac{a(1 - \gamma)}{1 - a \gamma}\right)^n \frac{r^n}{n + \beta} + (1 - a)\left(\frac{1}{\beta} - \frac{2 }{1 + \gamma} \sum_{n = 1}^{\infty} \frac{r^n}{n + \beta}\right)\\
& = (a - 1)\left( \frac{\gamma(a + 1)}{(1 - a \gamma)\beta} + \frac{2}{1 + \gamma}\sum_{n = 1}^{\infty} \frac{r^n}{n + \beta}\right) + \frac{(1 - a^2)}{a(1 - a \gamma)} \sum_{n = 1}^{\infty}\left(\frac{a(1 - \gamma)}{1 - a \gamma}\right)^n \frac{r^n}{n + \beta}.
\end{align*}
Letting $a \rightarrow 1$, we obtain
$$
M_{a, \gamma, \beta}(r) = O ((1 - a)^2).
$$
Further from (\ref{niru-vasu-P8-eq021}), we obtain
$
\left(\frac{1}{\beta}  - \frac{2}{1 + \gamma} \sum_{n = 1}^{\infty}\frac{r^n}{(n + \beta)}\right) \geq 0
$
for all $r \leq R_{\gamma, \beta}$. Therefore
$$
\frac{1}{\beta} - \frac{2 }{1 + \gamma} \sum_{n = 1}^{\infty} \frac{r^n}{n + \beta} < 0 \quad \mbox{for }~~~ r > R_{\gamma, \beta}.
$$
These two facts show that $R_{\gamma, \beta}$ cannot  be improved.

\end{proof}

\section{Acknowledgement} The research of the first author is supported by SERB-CORE grant. The research of the second author is supported by the NBHM post-doctoral fellowship, Department of Atomic Energy (DAE), Government of India  (File No: $0204/16(20)/2020/\mbox{R}\&\mbox{D}-II/10)$.

\end{document}